\documentclass[12pt]{article} %{amsart}
\usepackage{amsmath}
\usepackage{amssymb}
\usepackage{amsthm}
\usepackage[top=1in, bottom=1in, left=1in, right=1in]{geometry}
\usepackage{xcolor}
\usepackage[backref,pagebackref,pdftex,hyperindex]{hyperref}
\usepackage{comment}

\makeatletter
\newcommand{\leqnomode}{\tagsleft@true\let\veqno\@@leqno}
\makeatother

\newtheorem{theorem}{Theorem}[section]
\newtheorem{lemma}[theorem]{Lemma}
\newtheorem{prop}[theorem]{Proposition}
\newtheorem{coro}[theorem]{Corollary}

\newtheorem{example}[theorem]{Example}

\newtheorem{remark}[theorem]{Remark}

\newcommand{\PSH}{\textup{PSH}}

\newcommand{\vp}{\varphi}
\newcommand{\R}{\mathbb{R}}

\title{Lines in the space of K\"ahler metrics}
\author{Tam\'as Darvas, Nicholas McCleerey}
\date{}

\begin{document}
\maketitle
\begin{abstract}
We establish a Ross–Witt Nystr\"om correspondence for weak geodesic lines in the (completed) space of K\"ahler metrics. We construct a wide range of weak geodesic lines on arbitrary projective K\"ahler manifolds that are not generated by holomorphic vector fields, in the process disproving a folklore conjecture popularized by Berndtsson. Remarkably, some of these weak geodesic lines turn out to be smooth. In the case of Riemann surfaces, our results can be significantly sharpened. Finally, we investigate the validity of Euclid’s fifth postulate for the space of K\"ahler metrics.
\end{abstract}

\section{Introduction}

Suppose that $(X,\omega)$ is a K\"ahler manifold. By $\mathcal H_\omega \subset C^\infty(X)$ we denote the space of K\"ahler potentials:
$$\mathcal H_\omega := \{u \in C^\infty \ |\  \omega_u:=\omega + i \partial \bar \partial u > 0\}.$$

The search for a canonical representative in $\mathcal{H}_\omega$ has a long and rich history, tracing back to Calabi’s influential conjectures \cite{Cal} and the celebrated theorems of Yau \cite{Yau}. In connection with this program, it has become clear that one must equip the space $\mathcal{H}_\omega$ with special $L^p$-type Finsler geometries, with the choice of $p$ varying depending on the context and specific problem considered \cite{Da15}:
$$\| \phi\|_{L^p} = \bigg(\int_X | \phi|^p \omega_u^n \bigg)^{\frac{1}{p}}, \ \ \ u \in \mathcal H_\omega, \  \phi \in C^\infty(X) \simeq T_u \mathcal H_\omega.$$

The case $p=2$ is the classical Riemannian metric of Mabuchi, Semmes, and Donaldson \cite{Ma,Se,Do}, relevant in the study of the Calabi flow \cite{CC, St, BDL1}, while the case $p=1$ has become increasingly relevant in the study of stability questions \cite{BBJ, BDL2, CC1,CC2}. By \cite[Theorem 2]{Da15} the $d_p$-metric completions of $\mathcal H_\omega$ coincide with $(\mathcal E^p_\omega,d_p)$, the finite energy spaces introduced by Guedj-Zeriahi \cite{GZ} (building on work of Cegrell in the local case \cite{Ce98}).

By $\PSH_\omega$ we denote the space of $\omega$-psh functions on $X$ (see \eqref{eq:PSH_def} below). For $1 \leq p \leq p'$, recall the following inclusions \cite{GZ}:
$$L^\infty \cap \PSH_\omega \subset \mathcal E^{p'}_\omega \subset \mathcal E^p_\omega \subset \PSH_\omega.$$ 
It is well known that the $(\mathcal E^p_\omega,d_p)$ are geodesic metric spaces, with various geometric convexity properties \cite{Da15,CC3,DL19}.

Fix $p\geq 1$. We are interested in studying (weak) {\it $L^p$ geodesic lines}, by which we mean families $\mathbb{R} \ni t \to u_t \in \mathcal E^p_\omega$ such that for any $t,t' \in \R$:
\begin{equation}\label{eq: d_p_geod}
d_p(u_t,u_{t'}) = |t-t'|d_p(u_0,u_1).
\end{equation}
We also insist that $u(x, s):= u_{\textup{Re }s}(x) \in \textup{PSH}_{\pi^*\omega}(X \times \Bbb C)$, where $\pi: X \times \Bbb C \to X$ is the projection onto the first component, to rule out anomalies which arise in the case of $p=1$ (see \cite[Theorem 1.2]{DL19} and the discussion preceding it).

We will say that a geodesic line $t \to u_t$ is  \emph{bounded} if  $u_t \in L^\infty \cap \textup{PSH}_\omega$. As pointed out in \cite{Da15}, bounded geodesic lines are $L^p$ geodesic lines for all $p \geq 1$. What is more, bounded lines satisfy the following global degenerate complex Monge--Amp\`ere equations:
$$(\pi^* \omega + i \partial \bar \partial u)^{n+1} = 0  \ \textup{ on } \ X \times \Bbb C.$$
So from a purely analytic point of view, in the case of bounded potentials, our goal in this paper is to characterize solutions to the above global partial differential equation.

\paragraph{Ross--Witt Nystr\"om correspondence for lines.}Our first result will provide a characterization of geodesic lines in terms of their Legendre transform, in the spirit of the various Ross--Witt Nystr\"om correspondences for geodesic rays (i.e. families as above but for $t\in [0, \infty)$)  \cite{RWN14,Da17,DDL3,DX22,HTX}. 
Roughly speaking, such correspondences aim to characterize the geodesic condition, by the extremal properties of their time Legendre transform. See Section \ref{sec: prelim} for versions of this result for rays that will be needed in this work.

Let $t \to u_t$ be a subgeodesic line. By definition this just means that $u(x, s):= u_{\textup{Re }s}(x) \in \textup{PSH}_{\pi^*\omega}(X \times \Bbb C)$. Due to plurisubharmonicity, for any $x \in X$, $t \to u_t(x)$ will be finite and convex (or identically $-\infty$); thus we can consider the following $t$-Legendre transform:
\begin{equation}\label{eq: t-Leg_trns_def}
\hat u_\tau := \inf_{t \in \Bbb R} (u_t - t \tau), \ \tau \in \Bbb R.
\end{equation}
By the Kiselman minimum principle, for fixed $\tau$, either $\hat u_\tau \in \textup{PSH}_\omega$ or $\hat u_\tau \equiv -\infty$ \cite{Ki}. Moreover, basic properties of the Legendre transform show that $\Bbb R \ni \tau \to \hat u_\tau(x) \in [-\infty,\infty)$ is usc and concave for any $x \in X$.

Attempting to define the dual notion to a geodesic line, we will say that a curve 
$$\Bbb R \ni \tau \to v_\tau \in \PSH_\omega \cup \{-\infty\}$$ 
is a \emph{test line} if for any fixed $x \in X$, $\Bbb R \ni \tau \to v_\tau(x) \in [-\infty,\infty)$ is usc and concave, and moreover there exists $C>0$ such that $v_\tau \equiv -\infty$ for any $|\tau| \geq  C$.

It is not too difficult to see that the Legendre transform sends subgeodesic lines (of sublinear growth) to test lines (Proposition \ref{prop: RWN_sublinear_subgeod_lines}); the inverse is given by the inverse Legendre transform, which is defined on a test line $\tau \to v_\tau$ by:
$$\check v_t := \sup_{\tau \in \Bbb R} (v_\tau + t\tau), \ \ t \in \Bbb R.$$
We aim to identify the conditions on the test line $\tau \to v_\tau$ under which its inverse Legendre transform $t \to \check{v}_t$ defines a weak geodesic line. To do so, we need a few definitions.

For a test line $\tau \to v_\tau$ we introduce the maximal interval $(\tau_v^-,\tau_v^+) \subset \Bbb R$ where the potentials $v_\tau$ are not identically equal to $-\infty$:
$$\tau^-_v:= \inf\{\tau \in \Bbb R\ |\  v_\tau \not \equiv -\infty\}, \ \ \ \tau^+_v:= \sup\{\tau \in \Bbb R\ |\ v_\tau \not \equiv -\infty\}.$$

We then say that the test line $\tau \to v_\tau$ is $L^p$ if $\check v_0 = \sup_\tau v_\tau \in \mathcal E^p_\omega$; 
that it is \emph{bounded} if $\check v_0 \in L^\infty \cap \textup{PSH}_\omega$; and that it is \emph{zero mass} if $\int_X \omega^n_{v_\tau} = 0$ for any $\tau \in (\tau_v^-,\tau_v^+)$. Here and throughout the paper the measures $\omega^n_{v_\tau}$ are the non-pluripolar complex Monge--Amp\`ere measures of Guedj--Zeriahi \cite{GZ}.

\begin{theorem}\label{thm: RWN_lines}The map $\{u_t\}_t \to \{\hat u_\tau\}_\tau$ between $L^p$ (bounded) geodesic lines and  zero mass $L^p$ (bounded) test lines is bijective, with inverse $\{v_\tau\}_\tau \to \{\check v_t\}_t.$
\end{theorem}

By combining Theorem \ref{thm: RWN_lines} with Proposition \ref{prop: turn_RWN} below, it will follow that zero mass test lines are automatically maximal (see immediately before Proposition \ref{prop: turn_RWN} for the definition).

Although \cite{WN24} does not explicitly discuss geodesic lines, the injectivity in Theorem 1.1 can be deduced from the formula in \cite[Theorem 3.5]{WN24} under the additional assumption that all $\hat{v}_\tau$ are locally bounded outside a fixed analytic subset of $X$. However, this regularity property typically fails for the Legendre transform of a general weak geodesic line. Consequently, to establish Theorem 1.1, we instead rely on the framework of relative pluripotential theory developed in \cite{DDL2, DDL4}, with the volume diamond inequality from \cite{DDL5} playing a crucial role.

Theorem \ref{thm: RWN_lines} has several consequences. First, it immediately implies that $L^p$ geodesic lines are always $t$-Lipschitz; in particular, if $t \to u_t$ is an $L^p$ geodesic line such that $u_{t'}\in L^\infty$ for a fixed $t'\in\mathbb{R}$, then $u_t\in L^\infty$ for all $t\in\mathbb{R}$. This intriguing property has a `negative' consequence: the second postulate of Euclid must fail in $\mathcal E^p_\omega,$ $p \geq 1$. Indeed, there can be no weak line passing through a bounded and an unbounded element of $\mathcal E^p_\omega$, i.e., the $d_p$-geodesic segment connecting such potentials can not be extended to a line (c.f. \cite{Da21}).

\paragraph{Lines on Riemann surfaces.} In the case of a Riemann surface $(X,\omega)$, due to the Monge-Amp\`ere operator coinciding with the Laplacian, one can prove a much more precise Ross-Witt Nystr\"om correspondence for bounded lines:

\begin{theorem}\label{thm: class_RS}
Suppose that $(X,\omega)$ is a Riemann surface, and $t \to u_t$ is a bounded geodesic line. Then the Legendre transform \eqref{eq: t-Leg_trns_def} of $u_t$ is given by 
$$\hat u_\tau = \frac{\tau - \tau_v^-}{\tau_v^+ - \tau_v^-}\hat u_{\tau_v^+} + \frac{\tau_v^+ - \tau}{\tau_v^+ - \tau_v^-}\hat u_{\tau_v^-} + g(\tau),  \ \tau \in [\tau_v^-, \tau_v^+].$$ 
for some continuous concave function $g: [\tau_{v}^-, \tau_{v}^+] \to \Bbb R$, and $\hat u_\tau \equiv -\infty$ otherwise.

Conversely, any bounded subgeodesic curve $t \to u_t$, whose Legendre transform is as above, must be a bounded geodesic line.
\end{theorem}

As evidenced by Example \ref{ref: not_Riemann}, this same result does unfortunately not hold in higher dimensions.

\begin{example}By the above result, bounded geodesic lines on Riemann surfaces are completely determined by the data
\[
(\tau_v^-,\tau_v^+, \hat u_{\tau_v^-}, \hat u_{\tau_v^+}, g),
\]
where $\hat u_{\tau_v^-}, \hat u_{\tau_v^+} \in \PSH_\omega$ are zero-mass potentials such that $\max(\hat u_{\tau_v^-}, \hat u_{\tau_v^+}) \in L^\infty(X)$. 
In this setting, the zero-mass condition means that the `Laplacian currents' $\omega_{\hat u_{\tau_v^-}}$ and $\omega_{\hat u_{\tau_v^+}}$ are supported on polar sets $P_-$ and $P_+$, respectively. 

The boundedness condition $\max(\hat u_{\tau_v^-}, \hat u_{\tau_v^+}) \in L^\infty(X)$ is automatically satisfied whenever $\overline{P_-} \cap \overline{P_+}$ is empty.
Since the Laplace equation on $X$ can be solved with a polar measure as right-hand side, the above result yields an abundant supply of bounded geodesic lines.
\end{example}

\paragraph{Examples on projective manifolds.} Using the criterion of Proposition \ref{prop: gen_crit}, Theorem \ref{thm: RWN_lines} produces explicit Fubini–Study type geodesic lines on projective manifolds which, somewhat surprisingly, have been overlooked in the literature.

\begin{example}\label{ex: proj}
Let $L \to X$ be an ample line bundle such that $c_1(L) = \{\omega\}$. Let $D_0,\ldots, D_n$ be divisors in the class $c_1(L)$ such that $\cap_j D_j = \emptyset$. Let $\varphi_j \in \textup{PSH}_\omega$ such that $\omega_{\varphi_j} = [D_j]$. 
\begin{equation}\leqnomode
u_t := 
\max(\varphi_0, \varphi_1 + t,\ldots, \varphi_n + t) \in \textup{PSH}_\omega \cap C^{0,1}, \ t \in \Bbb R \tag{i}
\end{equation}
defines a $C^{0,1}$ geodesic line.% \medskip\\
\begin{equation}\leqnomode
v_t := \log ( e^{\varphi_0} + \sum_{j =1}^n  e^{\varphi_j + t}) \in \textup{PSH}_\omega \cap C^\infty, \ t \in \Bbb R\tag{ii}
\end{equation}
defines a $C^{\infty}$ geodesic line, which is parallel to $t\to u_t$.
\end{example}

We say that two geodesic lines $t \to u_t,v_t$ are \emph{parallel} if $t \to d_p(u_t,v_t)$ is bounded on $\Bbb R$ \cite{DL20}. Recall that the function $t \to d_p(u_t,v_t)$ is convex due to \cite[Theorem 1.5]{CC3}, hence $t \to d_p(u_t,v_t)$ is constant if and only if it is bounded.

\begin{remark}
By possibly increasing the power of the ample line bundle $L$, divisors $D_j$ as in the statement of the theorem can always be found. Moreover, by Bertini's theorem, a generic choice of divisors $D_0,\ldots, D_n$  in the class $c_1(L)$ satisfies $\cap_j D_j = \emptyset$ \cite[Lemma 8.9, p. 387]{Dem12}. Hence plenty of geodesic lines can be constructed on any Hodge K\"ahler manifold $(X,\omega)$.
\end{remark}

In the proof of both (i) and (ii) we will show that the Legendre duals of  $t \to u_t$ and $t \to v_t$ are bounded zero mass test lines, thus concluding via Theorem \ref{thm: RWN_lines}. That being said, given the basic nature of these examples, an alternative direct argument can also be given for why they are lines. We direct readers interested in this to the discussion directly following Corollary \ref{cor: proj_2}(ii). 

These examples  disprove a folklore conjecture in the field (mentioned in the abstract of \cite{Bern22}), which asserted that weak geodesic lines should be induced by the one-parameter action of a holomorphic vector field $V$ for which $i_V \omega$ is $\bar \partial$-exact; see \cite[Example 4.26]{Sze_book} for the construction. By \cite[Theorem 4.1]{Bern22}, all smooth \emph{regular} geodesic lines are of this form; we thank Song Sun for showing us an alternative argument for this fact based on \cite[Proposition 4.2.13]{CT}. Moreover, if $X$ is toric, then all toric geodesic lines on $X$ 
%on toric Kähler manifolds
are also induced by vector fields \cite{Gu99}, lending further support to this conjecture. However, the examples in Example \ref{ex: proj} will not generally come from this construction, e.g. if $X$ does not admit any non-trivial holomorphic vector fields at all. 

\smallskip
The geodesic lines $t \to v_t$ constructed in Example \ref{ex: proj}(ii) are $C^\infty$, but typically not regular in the sense that $\omega_{v_t} \geq 0$ may not be K\"ahler, i.e., $t \to v_t$ is a smooth weak geodesic line. In light of \cite[Theorem 4.1]{Bern22}, regularity emerges as the essential condition enabling such theorems to hold.

\smallskip
Recall that $C^{1,1}$ regularity is optimal for weak geodesic segments joining smooth potentials \cite{Ch,CTW,DL12}. See \cite{Bl1} for complementary estimates and \cite{Bl2} for a comprehensive survey. In accordance with this, prior to our work, very few examples of segments or rays with higher regularity were known (cf.\ \cite{Do,AT03,Sun10}). Example~\ref{ex: proj} shows that, at least in the projective setting, smooth weak geodesic lines arise naturally and in large numbers.

\paragraph{Euclid's fifth postulate.}Earlier we pointed out that Euclid's second postulate fails in $(\mathcal E^p_\omega, d_p)$. Our last result concerns the validity of the fifth postulate, guaranteeing existence of parallel geodesic lines passing through outside points.

\begin{theorem}\label{thm: 5th_post} Let $p \geq 1$. Suppose that $\Bbb R \ni t \to v_t \in \PSH_\omega \cap L^\infty$ is a weak geodesic line and $u \in \textup{PSH}_\omega \cap L^\infty$. There exists a unique geodesic line $\Bbb R \ni t \to u_t  \in \PSH_\omega \cap L^\infty$ with $u_0 = u$ and $d_p(u_t,v_t) = d_p(u_0,v_0), \ t \in \Bbb R$ if and only if
\begin{equation}\label{eq: cond_5th_post}
u = \sup_{\tau \in \Bbb R} P[\hat v_\tau](u).
\end{equation}
Moreover, if \eqref{eq: cond_5th_post} holds, then $\hat u_\tau = P[\hat v_\tau](u)$ for all $\tau \in \Bbb R$.
\end{theorem}
Potentials on the right hand side of \eqref{eq: cond_5th_post} are called envelopes with respect to singularity type, and their definition is recalled in \eqref{eq: sing_env_def} below. 
Unfortunately, a typical potential $u \in \PSH_\omega \cap L^\infty$ does not satisfy \eqref{eq: cond_5th_post}, so that Euclid's fifth postulate fails (see the discussion after Theorem \ref{thm: 5th_post_1}). That being said, there exists a very convenient way to construct parallel lines to any $L^p$ geodesic line $t \to v_t$:

\begin{example}\label{ex: paral_line}
Let $g:[\tau^-_v,\tau^+_v] \to \Bbb R$ be a continuous concave function. Then it is straightforward to verify that $\tau \to w_\tau:= \hat v_\tau + g(\tau)$ is an $L^p$ zero mass test line, implying that $t \to \check w_t$ is an $L^p$ geodesic line.  Since $\check w_t - v_t$ is  uniformly bounded, so is $d_p(\check w_t,v_t)$, due to \cite[Theorem 5.5]{Da15}. 
Due to Buseman convexity of $\mathcal E^p_\omega$, the function $t \to d_p(\check w_t,v_t), \ t \in \Bbb R$ is convex, hence constant.  As a result, $t \to \check w_t$ is parallel to $t \to v_t$.
\end{example}

Due to Theorem \ref{thm: class_RS}, in the case of Riemann surfaces the above set of examples exhausts all bounded geodesic lines parallel to $t \to v_t$.

\paragraph{Flat embeddings and spherical buildings.}
Let $\varphi_j$ be as in Example \ref{ex: proj}. 
Using Corollary \ref{cor: proj_2} and its proof, it is not hard to show that the maps
\begin{equation}\label{eq: flat_maps}
 \Bbb R^n \ni (t_1,\ldots, t_n) \to \max(\varphi_0,\varphi_1+t_1,\ldots, \varphi_n + t_n) \in \PSH_\omega \cap C^{0,1}
 \end{equation}
are flat embeddings. 
In the projective case this allows one to have a plethora of flat embeddings of $\Bbb R^n$ into the space of (degenerate) K\"ahler metrics.

Donaldson conjectured that the geometry of the space of Kähler metrics should admit an infinite-dimensional spherical building structure, as a certain large scale $k$-limit of the spherical buildings attached to $SL(k)/SU(k)$ \cite{Do}
(see also \cite{Co}, \cite{Od15}, \cite{RWN23}). In the finite-dimensional setting of $SL(k)/SU(k)$,  the structure of the different flat embeddings of  $\Bbb R^k$ and their intersections underpin the associated spherical building.

While a precise definition of a building structure in the infinite-dimensional space of Kähler metrics remains elusive, several works have speculated on this using different perspectives. In \cite{RWN23}, the authors construct infinite-dimensional cones in the space of Kähler potentials arising from filtrations. Earlier, \cite{Od15, Co} interpreted the space of test configurations as a direct limit of Tits buildings, and investigated the asymptotic behavior of certain invariants along this limit. For connections with the non-Archimedean point of view we refer to \cite{BoEr,BoJ}.

To maintain a clear focus on geodesic lines and for the sake of brevity, we do not pursue the study of higher-dimensional flats and buildings in the present work. Building on the results and techniques of this work, we intend to return to this in a future correspondence.

\paragraph{Acknowledgments.} We thank B. Berndtsson, R. Berman, H. Guenancia, E. Di Nezza, L. Lempert, S. Sun, K. Zhang   for valuable conversations related to  geodesic lines. We thank D. Witt Nystr\"om for conversations related to the flat embeddings of \cite{RWN23} and beyond. We also thank V. Guedj for pointing out that our techniques allow to construct smooth weak geodesic lines, not just $C^{0,1}$ ones, as elaborated in Corollary \ref{cor: proj_2}(ii),  and the discussion following its proof.

\section{Preliminaries}\label{sec: prelim}

In this preliminary section we recall basic facts and prove a slightly novel version of the Ross--Witt Nystr\"om correspondence for geodesic rays that we will need later.
We fix a K\"ahler manifold $(X,\omega)$ with total volume $V = \int_X \omega^n$ for the whole paper. 

Given a complex manifold $Y$ and closed smooth $(1,1)$-form $\eta$, we denote by 
\begin{equation}\label{eq:PSH_def}
\PSH_\eta
\end{equation}
the space of functions $v: Y \to [-\infty,\infty)$ that are locally the difference of a plurisubharmonic (psh) and a smooth function, such that $\eta_v := \eta + i \partial \bar \partial  v \geq 0$ in the sense of currents. Such functions are automatically upper semi-continuous (usc) and locally integrable. For more details we refer to \cite{GZ}.

\paragraph{Subgeodesics and test curves.}

Fix $u_0\in\PSH_\omega$. We say that $[0,\infty) \ni t \to u_t \in \textup{PSH}_\omega$ is a \emph{subgeodesic ray}  emanating from $u_0$, if $u_{t} \to u_0$ in weak $L^1$ as $t \to 0$, and  $u(x,z) := u_{\textup{Re }z}(x) \in \PSH_{\pi^*\omega}(X\times \{\textup{Re }z >0\})$, where $\pi: X\times \Bbb C \to X$ is the projection onto the first component.

A subgeodesic ray is \emph{sublinear} if there exists $C >0$ such that $u_t(x) \leq C t + C$ for all $t\geq 0$, $x \in X$.

A \emph{psh geodesic ray} is a sublinear subgeodesic ray that additionally satisfies the following maximality property: for any $0 \leq  a < b$, the subgeodesic $(0,1) \ni t \mapsto v^{a,b}_t:=u_{a(1-t) + bt} \in \PSH_\omega$ can be recovered in the following manner:
\begin{equation}\label{eq: vabt_eq}
v^{a,b}_t:=\sup_{h \in \mathcal{S}}{h_t}\,,\quad t \in  (0,1)\,,
\end{equation}
where $\mathcal{S}$ is the set of subgeodesics $(0, 1) \ni  t \to h_t \in \PSH_\omega$ such that $\lim_{t \searrow 0}h_t\leq  u_a$ and $\lim_{t \nearrow 1}h_t\leq u_b$ in weak $L^1$.

\smallskip Now we turn to the dual notion of test curves. Making small tweaks to \cite[Definition 5.1]{RWN14} and \cite[Definition 3.4]{DZ24}, a map $\mathbb{R}\ni \tau \mapsto \psi_\tau\in \PSH_\omega \cup\{-\infty\}$ is a \emph{psh test curve} if $\tau\mapsto \psi_\tau(x)$ is concave, decreasing, and usc for any $x\in X$; $\psi_\tau\equiv -\infty$ for all $\tau$ big enough; and $\psi_\tau \nearrow \psi_{-\infty} \in \textup{PSH}_\omega$ a.e.

For the rest of this note, we adopt the following convention: test curves/lines will always be parameterized by $\tau$, whereas geodesic rays/lines will be parametrized by $t$.  Consequently, $\{u_t\}_t$ (or $t \to u_t$) will always denote some kind of geodesic, whereas $\{v_\tau\}_\tau$ (or $\tau \to v_\tau$) will denote some type of test curve. Thus, the parameters $t$ and $\tau$ can also be regarded as dual to one another.

To any test curve $\tau \to \psi_\tau$ we also attach the following constant:
\begin{equation}\label{eq: tau_ray}
\tau_\psi^+ := \inf\{\tau \in \mathbb{R} : \psi_\tau \equiv -\infty\}\,.
\end{equation}

Recall from \cite{RWN14,DDL2} that for $\psi,\chi \in \textup{PSH}_\omega$ the envelope of the singularity type of $\chi$ with respect to $\psi$ is defined as
\begin{equation}\label{eq: sing_env_def}
P[\chi](\psi) : = \textup{usc}_X \big( \sup \{ v \in \textup{PSH}_\omega, v \leq \psi, \ \ v \leq \chi + C, \textup{ for some }C \in \Bbb R\}\big);
\end{equation}
here, $\textup{usc}_X$ is the upper semicontinuous regularization of a function on $X$.

A psh test curve $\tau \to \psi_\tau$  is \emph{maximal} if $P[\psi_\tau](\psi_{-\infty}) =\psi_\tau$ for any $\tau \in \mathbb{R}$, where we follow the convention $P[-\infty](\psi_{-\infty}) =-\infty$.

Given a test curve $\tau \to \psi_\tau$, its inverse Legendre transform is defined as
\begin{equation*}\label{eq: inverse_Lag_tran_def}
\check{\psi}_t(x):=\sup_{\tau \in \mathbb{R}} (\psi_\tau(x)+t\tau)\,, \quad t > 0,\, x \in X.
\end{equation*}

Given a subgeodesic ray $t \to \phi_t$, its Legendre transform is defined as 

 \begin{equation*}
 \hat \phi_\tau(x):= \inf_{t > 0} (\phi_t(x)-t\tau)\,,\quad \tau\in \mathbb{R},\, x \in X.
 \end{equation*}

The following result is proved in \cite[Theorem 3.7(ii)]{DX22} for $u =0$, and the argument for the more general result is exactly the same, requiring no change:

\begin{theorem} \label{thm: max_test_curve_ray_duality} Let $u \in \textup{PSH}_\omega$.
        The (inverse) Legendre transform $\{\psi_\tau\}_\tau \mapsto \{\check \psi_t\}_t$ gives a bijective map between maximal test curves with $\psi_{-\infty} = u$ and psh geodesic rays for which $\check \psi_0 = u$.  The inverse map is $\{\phi_t\}_t \mapsto \{\hat \phi_\tau\}_\tau$. 
\end{theorem}

\paragraph{Bounded and finite energy rays.}

We discuss the same objects in the finite energy and bounded context. We say that a subgeodesic ray $t \to u_t$ is $L^p$ (resp. \emph{bounded}) if $u_t \in \mathcal E^p_\omega$ (resp. $u_t \in L^\infty$) for all $t \geq 0.$

By \cite[Theorem 2]{Da15}, an $L^p$ subgeodesic ray $t \to u_t$ is a psh geodesic ray if and only if it is an $L^p$ geodesic ray in the sense of \eqref{eq: d_p_geod} for $t,t' \in [0,\infty)$.

A test curve $\tau \to \psi_\tau$ is  \emph{bounded} if $\psi_{-\infty}\in L^\infty$ and there exists $\tau' \in \R$ such that $\psi_\tau = \psi_{-\infty}$ for all $\tau \leq \tau'$.

A test curve $\tau \to \psi_\tau$ has  \emph{finite energy} (or is $L^1$) if $\psi_{-\infty} \in \mathcal E^1$ and 
    \begin{equation*}\label{eq: fetestcurve_def}
        \int_{-\infty}^{\tau^+_\psi} \left( \int_X \omega_{\psi_\tau}^n-\int_X \omega_{}^n \right) \,\mathrm{d}\tau >-\infty\,.
    \end{equation*}
Recall that the Monge--Amp\`ere energy $I: \mathcal E^1_\omega \to \Bbb R$ is defined as
$$I(u):=\frac{1}{V} \sum_{j=0}^n u \omega_u^j \wedge \omega^{n-j}, \ u \in \mathcal E^1_\omega.$$
It is well known that $I$ is convex along $L^1$ subgeodesics and is affine precisely along $L^1$ (weak) geodesics \cite[Theorem 6.2]{BBGZ}.
As a result, for any $L^1$ subgeodesic ray $t \to u_t$ it is possible to define its Monge--Amp\`ere slope:
\[
I\{\phi_t\}:= \lim_{t \to \infty} \frac{I(\phi_t)}{t} \in \Bbb R \cup \{\infty\}.
\]

As we will see soon, the Legendre transform maps finite energy (resp. bounded) geodesic rays bijectively onto the space of maximal finite energy (resp. bounded) test curves (Theorem \ref{thm: max_fin_en_test_curve_ray_duality} below). Before doing so, we recall one additional concept which we will need in the proof.

\paragraph{Maximization of test curves.} Given a test curve $\tau \to \psi_\tau$, and a function $v\in\PSH_\omega$, by taking the envelope of singularity type with respect to $v$ along the curve, it possible to produce a maximal test curve $\{\psi^v_\tau\}$ with $\psi^v_{-\infty} = v$; we refer to this as the \emph{maximization} of $\{\psi_\tau\}$ (at $v$). For simplicity, we will only consider the case of $\psi_{-\infty}, v \in \mathcal{E}^1_\omega$ with $\psi_{-\infty} \leq v$.

\smallskip
First, we introduce $w_\tau := P[\psi_\tau](v), \ \tau \in \Bbb R$. Since $\tau \to \psi_\tau$ is $\tau$-concave, so is $\tau \to w_\tau$. We claim that $w_\tau \nearrow v$ a.e. as $\tau \to -\infty$. To show this, note that by \cite[Theorem 3.8]{DDL2}: 
\begin{equation}\label{eq: DDL2ineq}
\omega_{w_\tau}^n = \omega_{P[\psi_\tau](v)}^n \leq \textbf{1}_{\{ v = w_\tau\}} \omega_v^n \leq \omega_v^n, \ \tau < \tau^+_\psi.
\end{equation}
Due to \cite[Theorem 2.3]{DDL2} we have that  $\int_X \omega_{\psi_\tau}^n \nearrow \int_X \omega_{\psi_{-\infty}}^n$. Since $\int_X \omega_{\psi_\tau}^n = \int_X \omega_{w_\tau}^n$ (\cite[Theorem 1.3]{DDL2}) and $\omega_{\psi_{-\infty}}^n=\int_X \omega_v^n$, we obtain that $\int_X \omega_{w_\tau}^n \nearrow \int_X \omega_v^n$.

Let $\psi_{-\infty} \leq w_{-\infty} := \lim_{\tau \to -\infty} w_\tau \leq v$. Since $\psi_{-\infty} \in \mathcal E^1_\omega$, $w_{-\infty} \in \mathcal E^1_\omega$ also. From \cite[Theorem 2.3, Remark 2.5]{DDL2} and \eqref{eq: DDL2ineq} we obtain $\omega_{w_{-\infty}}^n \leq \omega_v^n$. Since both measures have the same total mass, we have $\omega_{w_{-\infty}}^n = \omega_v^n$.  Uniqueness of complex Monge--Amp\`ere measures \cite{Din} and the fact that $w_{-\infty}$ and $v$ agree at least at one point gives $w_{-\infty} = v$, proving the claim.

Now, since $\tau \to w_\tau$ may not be $\tau$-usc, we define the maximization of $\tau \to \psi_\tau$ to be:
\begin{equation}\label{eq: maxim_def}
\psi^v_\tau := w_\tau = P[\psi_\tau](v), \ \tau \neq \tau^+_\psi,\text{ and } \ \psi^v_{\tau^+_\psi} := \lim_{\tau \searrow \tau^+_\psi}\psi^v_\tau.
\end{equation}
That $\psi^v_\tau = P[\psi^v_\tau](v)$ follows from \cite[Proposition 3.11 and Theorem 3.8]{DDL2}. Hence $\tau \to \psi^v_\tau$ is a maximal test curve with $\psi^{v}_{-\infty}=v$; additionally, we have $\tau^+_{\psi^v} = \tau^+_{\psi}$. Finally, due to \cite[Theorem 1.3]{DDL2}, the energy of a test curve does not change under maximization:
    \begin{equation}\label{eq: fetestcurve_inv}
        \int_{-\infty}^{\tau^+_{\psi^v}} \left( \int_X \omega_{\psi^v_\tau}^n-\int_X \omega_{}^n \right) \,\mathrm{d}\tau=\int_{-\infty}^{\tau^+_\psi} \left( \int_X \omega_{\psi_\tau}^n-\int_X \omega_{}^n \right) \,\mathrm{d}\tau.
    \end{equation}
\begin{remark}\label{rem: ray_max} There is a corresponding maximization procedure for a finite energy subgeodesic ray $t \to \phi_t$ with $\phi_0 \leq v$. Let $l>0$, and consider the $L^1$ geodesic segments $[0,l] \ni t \to \phi^l_t$ connecting $\phi^l_0 := v$ and $\phi^l_l := \phi_l$. Using the comparison principle, one can show that for fixed $t$ the potentials $\{\phi^l_t\}_l$ are increasing; moreover their limit $t \to \phi^v_t:= \lim_l \phi^l_t$ yields an $L^1$ geodesic ray emanating from $v$ satisfying
\begin{equation}
I\{\phi_t\} = I\{\phi^v_t\}.\label{maxim_I_slope}
\end{equation}
It will follow from the next theorem that the two maximization operations are dual to each other via the Legendre transform. Additionally, in case $t \to \phi_t$ is an $L^1$ geodesic, the $L^1$ ray $t \to \phi^v_t$ is the $d_1$-parallel ray to $t \to \phi_t$ that emanates from $v$ (see \cite[Proposition 4.1]{DL20}).
\end{remark}

\paragraph{Finite energy Ross--Witt Nystr\"om correspondence.} We conclude this section with a Ross--Witt Nystr\"om correspondence between maximal finite energy test curves and finite energy rays; the case of bounded rays and test curves follows along similar lines. 

\begin{theorem} \label{thm: max_fin_en_test_curve_ray_duality} Let $u \in \mathcal E^1_\omega$.
        The Legendre transform $\{\psi_\tau\}_\tau \mapsto \{\check \psi_t\}_t$ gives a bijective map between maximal finite energy test curves with $\psi_{-\infty} = u$ and $L^1$ geodesic rays for which $\check \psi_0 = u$.  
        
        The inverse map is $\{\phi_t\}_t \mapsto \{\hat \phi_\tau\}_\tau$. For the Monge--Amp\`ere slope we additionally have the formula
        \begin{equation}\label{eq: I_RWN_form}
            I_\omega \{\check \psi_t\}=\frac{1}{V}\int_{-\infty}^{\tau^+_\psi} \left(\int_X \omega_{\psi_\tau}^n-\int_X \omega_{}^n \right) \,\mathrm{d}\tau+\tau_{\psi}^+\,.
        \end{equation}
\end{theorem}

This result is a slight generalization of \cite[Theorem 3.7(iv)]{DX22} that considers $u=0$ (c.f. \cite[Theorem 2.6]{DXZ23}) -- in fact, our argument is by reduction to the case $u=0$.

\begin{proof}Let $u_j \in \mathcal H_\omega$ such that $u_j \searrow u$ \cite{BK07}. We consider the maximal test curves  $\tau \to \psi^{u_j}_\tau$ constructed above. Since $\psi^{u_j}_\tau \searrow \psi_\tau$ for any $\tau \in \Bbb R$, by basic properties of the Legendre transforms we obtain that
$$\check \psi^{u_j}_t \searrow \check \psi_t, \ t \geq 0.$$
By \cite[Theorem 3.7(iv)]{DX22} the curves $t \to \check \psi^{u_j}_t$ are geodesic rays. Due to \eqref{eq: fetestcurve_inv} we have that 
        \begin{equation}\label{eq: I_RWN_form_2}
            I \{\check \psi^{u_j}_t\}
=\frac{1}{V}\int_{-\infty}^{\tau^+_\psi} \left(\int_X \omega_{\psi_\tau}^n-\int_X \omega_{}^n \right) \,\mathrm{d}\tau+\tau_{\psi}^+\,.
        \end{equation}
Consequently, for fixed $t>0$ and indices $j>k$ we have that $I(\psi^{u_j}_t) - I(\psi^{u_j}_0) =I(\psi^{u_j}_t) - I(u_j)= I(\psi^{u_k}_t) - I(\psi^{u_k}_0)=I(\psi^{u_k}_t) - I(u_k)$. As $\psi^{u_k}_t \leq \psi^{u_j}_t$ we have $d_1(u_k,u_j)=d_1(\psi^{u_k}_t,\psi^{u_j}_t) \to 0$ as $k,j \to \infty$. As a result, due to \cite[Proposition 4.3]{BDL1}, $t \to \lim_j \psi^{u_j}_t = \psi_t $ is a $d_1$-geodesic ray, as expected. Moreover taking the limit in \eqref{eq: I_RWN_form_2} the identity \eqref{eq: I_RWN_form} follows.

That $\{\phi_t\}_t \mapsto \{\hat \phi_\tau\}_\tau$ is the inverse map follows from Theorem \ref{thm: max_test_curve_ray_duality}.
\end{proof}

We conclude with a remark on the Monge--Amp\`ere slope of subgeodesic rays (c.f. \cite{DZ24}):
\begin{coro}\label{cor: maximiz_I_slope_for}  Let $t \to \phi_t$ be a sublinear $L^1$ subgeodesic ray. Then
$$I\{\phi_t\} = \frac{1}{V}\int_{-\infty}^{\tau^+_\phi} \left(\int_X \omega_{\hat \phi_\tau}^n-\int_X \omega_{}^n \right) \,\mathrm{d}\tau+\tau_{\phi}^+\,.$$
\end{coro}
\begin{proof}
 Let $t \to \phi_t$ be a sublinear $L^1$ subgeodesic ray and $t \to \psi_t$ the maximization of $t \to \phi_t$ with respect to $\phi_0$. By Theorem \ref{thm: max_fin_en_test_curve_ray_duality} and Remark \ref{rem: ray_max}, we have that
$$I\{\phi_t\} =I\{\psi_t\} = \frac{1}{V}\int_{-\infty}^{\tau^+_\psi} \left(\int_X \omega_{\hat \psi_\tau}^n-\int_X \omega_{}^n \right) \,\mathrm{d}\tau+\tau_{\psi}^+\,.$$
The result now follows from \eqref{eq: fetestcurve_inv} and the fact that $\tau^+_\psi = \tau^+_\phi$.
\end{proof}

\section{Lines and test lines}

We say that $\Bbb R \ni t \to u_t \in \textup{PSH}_\omega$ is a \emph{subgeodesic line}  if and only if $u(x,z) := u_{\textup{Re }z}(x) \in \PSH_{\pi^*\omega}(X\times \Bbb C)$. The subgeodesic line is $L^p$ (resp. \emph{bounded}) if $u_t \in \mathcal E^p_\omega$ (resp. $u_t \in L^\infty$) for all $t \in \Bbb R.$

Recall from the introduction that a curve $\Bbb R \ni \tau \to v_\tau \in \PSH_\omega$ is a \emph{test line} if for any fixed $x \in X$, $\Bbb R \ni \tau \to v_\tau(x) \in [-\infty,\infty)$ is usc and concave, and there exists $C>0$ such that $v_\tau \equiv -\infty$ for any $|\tau| \geq  C$. 
For a test line $\tau \to v_\tau$ we introduce the maximal (open) interval $(\tau_v^-,\tau_v^+)$ where the potentials $v_\tau$ are not identically equal to $-\infty$:
\begin{equation}\label{eq: tau_line}
\tau^-_u:= \inf\{\tau \in \Bbb R\ |\  v_\tau \not \equiv -\infty\}, \ \ \ \tau^+_u:= \sup\{\tau \in \Bbb R\ |\ v_\tau \not \equiv -\infty\}.
\end{equation}

Due to $\tau$-concavity and the $\tau$-usc property we have that $v_\tau \not \equiv -\infty$ for all $\tau \in (\tau^-_v,\tau^+_v)$. 
Before we can consider the Legendre transforms of test lines, we need two technical lemmas:

\begin{lemma}\label{lem: line_test_curve_sup} Let $\tau \to v_\tau$ be a test line. For any $t \in \Bbb R$, the potential $\textup{usc}_X\big(\sup_{\tau \in \Bbb R} v_\tau +t\tau\big) \in \PSH_\omega$ exists and is well defined. 
\end{lemma}
\begin{proof} Since $v_\tau \equiv -\infty$ for $\tau \not \in  [\tau^-_v,\tau^+_v]$, it is enough to show the result for $t=0$. The only concern is that the supremum of the $v_\tau$ might not be finite. For this it is enough to show existence of $c>0$ such that $v_\tau \leq c$ for any $\tau \in [\tau^-_v,\tau^+_v]$. 

From $\tau$-concavity we obtain that $\tau \to \int_X v_\tau \omega^n$ is concave (hence bounded from above) on the finite length interval $(\tau_v^-,\tau_v^+)$. The Hartogs lemma now gives that  $\sup_X v_\tau$ is also bounded  from above on $[\tau_v^-,\tau_v^+]$, finishing the proof.
\end{proof}

\begin{lemma}\label{lem: line_test_curve_sup1}  Let $\tau \to v_\tau$ be a test line.  The function $(t, x) \to \sup_{\tau \in \Bbb R}( v_\tau +t\tau)$ is usc on $\Bbb R \times X$.
\end{lemma}

The proof is almost the same as \cite[Lemma 3.8]{DX22}:

\begin{proof}One needs to show that on $\Bbb R \times X$ we have that 

$$\textup{usc}_{\Bbb R \times X}\big(\sup_{\tau \in \Bbb R} (v_\tau +t\tau)\big) = \sup_{\tau \in \Bbb R} (v_\tau +t\tau)$$

To start we consider the complexification of the inverse Legendre transform: 
\[
u(s,z) := \sup_{\tau} (v_\tau(z) +  \tau\textup{Re} s)\,,\quad (s,z) \in \Bbb C \times X\,.
\]
By definition, $u_t(x): = u(t,x) \leq u_0(x) +  |t|\max\{\tau^+_v, \tau_v^-\}, t  \in \Bbb R$. Clearly, $\textup{usc  } u \in \PSH_{\pi^*\omega}(\Bbb C\times X)$, where  $\pi:\Bbb C\times X\rightarrow X$ is the natural projection and $\textup{usc} = \textup{usc}_{\mathbb{C}\times X}$.

It will be enough to show that  $\textup{usc } u=u$. For this we introduce $E = \{ u < \textup{usc } u\} \subseteq \Bbb C \times X$. As both $u$ and $\textup{usc } u$ are $\mathbb{R}$-invariant in the imaginary direction of $\Bbb C$, it follows that $E$ is also $i \mathbb{R}$-invariant, i.e., there exists $B \subseteq \Bbb R \times X$ such that $E = B +i \mathbb R$.

As $E$ has Monge--Ampère capacity zero, it follows that $E$ has zero Lebesgue  measure. By Fubini's theorem $B \subseteq \Bbb R \times X$ has zero  Lebesgue measure as well. For $z \in X$, we introduce the slices:
\[
B_z = B \cap \left(\Bbb R \times \{z\}\right)\,.
\]
By Fubini's theorem again, we have that $B_z$ has zero Lebesgue measure for all $z \in X \setminus F$, where $F \subseteq X$ is some set of zero Lebesgue measure. 

By slightly increasing $F$, but not its zero Lebesgue measure(!), we can additionally assume that $u_t(z)> -\infty$ for all $t \in \Bbb R$ and $z \in X \setminus F$. Indeed, at least one potential $v_\tau$ is not identically equal to $-\infty$, so we can add the zero Lebesgue measure set $\{x \in X :   v_\tau(x) = -\infty\}$ to $F$. 

Let $z \in X \setminus F$. We argue that $B_z$ is in fact empty. By the  assumptions on $F$, the maps $t \mapsto u_t(z)$ and $t \mapsto (\textup{usc } u)(t,z)$ are locally bounded and convex (hence continuous) on $\Bbb R$. As they agree on the dense set $\Bbb R \setminus B_z$, it follows that they have to be the same, hence $B_z=\emptyset$, allowing to conclude that 
\begin{equation}\label{eq: a.e._id1}
\chi_\tau := \inf_{ t \in \Bbb R} [u_t(x) - \tau t] = \inf_{ t \in \Bbb R} \left[{(\textup{usc } u)}(t,x) - \tau t\right]\,, \quad \tau \in \mathbb{R}  \textup{ and } x \in X \setminus F\,.
\end{equation}
By duality of the Legendre transform $v_\tau(x) = \inf_{t \in \Bbb R} [u_t(x) - t\tau]$ for all  $x \in X$ and $\tau \in \mathbb R$ (here is where the $\tau$-usc property of $\tau \mapsto v_\tau$ is used). Using this  and \eqref{eq: a.e._id1} we obtain that $v_\tau=\chi_\tau$ on $X \setminus F$ for all $\tau \in \mathbb{R}$, hence $v_\tau=\chi_\tau$ a.e. on $X$. Since both $v_\tau$ and $\chi_\tau$ are $\omega$-psh (the former by definition, the latter by Kiselman's minimum principle \cite{Ki}), it follows that in fact $v_\tau \equiv\chi_\tau$ for all $\tau \in \mathbb{R}$, i.e., \eqref{eq: a.e._id1} holds for all $\tau \in \Bbb R, \ x \in X$. 

Let $x \in X$. Due to $t$-convexity,  $t \to u_t(x)$ is either identically equal to $-\infty$, or finite and convex on $\Bbb R$. The same is true for $t \to \textup{usc } u(t,x)$. Thus, applying the $\tau$-Legendre transform to the $\tau$-usc and $\tau$-concave functions $\tau \to v_\tau(x)$ and $\tau \to \chi_\tau(x)$, we conclude that $u_t(x) = \textup{usc } u(t,x)$ for all $(t,x)\in \Bbb R \times X$.
\end{proof}

\paragraph{Sublinear subgeodesic lines and test lines.} We recall the Legendre transform of a subgeodesic line $t \to u_t$ and the inverse Legendre transform of a test line $\tau \to v_\tau$, that will give us correspondences between various types of (sub)geodesic lines and test lines:
$$\hat u_\tau := \inf_{t \in \Bbb R} (u_t -t\tau), \ \ \ \check v_t := \sup_{\tau \in \Bbb R} (v_\tau + t\tau).$$

We say that a subgeodesic line $t \to u_t$ is \emph{sublinear} if there exists $c>0$ such that $u_t(x) \leq c + c|t|.$ With this notion in hand, we are ready to prove the simplest form of our Ross--Witt Nystr\"om correspondence:

\begin{prop}\label{prop: RWN_sublinear_subgeod_lines} The maps $\{u_t\}_t \to \{\hat u_\tau\}_\tau$ and $\{v_\tau\}_\tau \to \{\check v_t\}_t$ give bijective correspondences between sublinear subgeodesic lines and test lines.
\end{prop}

\begin{proof} If $t \to u_t$ is a sublinear subgeodesic line, then $\hat u_\tau \equiv -\infty$ for $|\tau| > c$. As $\tau$-concavity and $\tau$-upper semi-continuity follow from basic facts about Legendre transforms, we obtain that $\tau \to \hat u_\tau$ is a test line.

Similarly, if $\tau \to v_\tau$ is a test line, then using the previous two lemmas and Choquet's lemma we obtain that $t \to \check v_t$ is a sublinear subgeodesic line.
\end{proof}

Any subgeodesic line $t \to u_t$ is the concatenation of two subgeodesic rays:
\begin{equation}\label{eqn: conc_rays}
u^+_t := u_t, \ t \geq 0, \ \ \ \ u^-_t := u_{-t}, \ t \geq 0.
\end{equation}
Write $\tau \to \hat{u}^+_\tau$ and $\tau \to \hat{u}^-_\tau$ for the test curves corresponding to these geodesic rays. It is possible to link these to the Legendre transform of the subgeodesic line:

\begin{lemma} Suppose that $t \to u_t$ is a subgeodesic line. Then for any $\tau \in \Bbb R$ we have that $\hat{u}^+_\tau = \sup_{\sigma\geq \tau} \hat{u}_\sigma$ and $\hat{u}^-_\tau = \sup_{\sigma \leq -\tau} \hat{u}_\sigma$.
\end{lemma}

\begin{proof}We only argue the first formula as the second follows automatically after replacing $t \to u_t$ with $t \to u_{-t}$. 

Consider the curve:
\[
v_\tau := \sup_{\sigma\geq \tau} \hat{u}_\sigma \text{ for }\tau \in \Bbb R.\]
That $v_\tau \in \PSH_\omega$ follows from Lemma \ref{lem: line_test_curve_sup1} applied to the truncated test line $\sigma \to \hat u_\sigma $ where we replace potentials $\hat u_\sigma$ for $\sigma < \tau$ with $-\infty$.

That $\tau \to v_\tau$ is $\tau$-usc, $\tau$-concave, $\tau$-decreasing follows from basic properties of convex functions -- hence it is a test curve. To conclude the first formula of the lemma, it is enough to show that the dual Legendre transform of $\tau \to v_\tau$ is the ray $t \to u_t, \ t \geq 0$:
\begin{align*}
\check{v}_t &= \sup_{\tau\in \R} (v_\tau + t\tau) = \sup_{\tau\in \R} \sup_{\sigma \geq \tau} (\hat{u}_\sigma + t\tau)= \sup_{\sigma\in \R}\sup_{\tau \leq \sigma} (\hat{u}_\sigma + t\tau) \\
&= \sup_{\sigma\in \R} (\hat{u}_\sigma + t\sigma) = u_t,
\end{align*}
where we used that $t\geq 0$ in the last line. 
\end{proof}

It follows immediately that: 
\begin{lemma} \label{lem: tau_line_ray} For any test line $\tau\mapsto u_\tau$, we have  $\tau^+_{u^+} = \tau^+_{u}$ and $\tau^+_{u^-} = -\tau^-_{u}$. 
\end{lemma}

\paragraph{Geodesic turns and maximal test lines.} 
We say that an $L^p$ subgeodesic line $t \to u_t$ is an $L^p$ \emph{geodesic turn} (or simply an $L^p$ turn) if the subgeodesic rays $u^+_t$ and $u^-_t$ defined in \eqref{eqn: conc_rays} are actually geodesic rays.
We define \emph{bounded turns} similarly.

\begin{lemma} If $t \to u_t$ is an $L^p$ turn then it is sublinear. In particular, if $u_0$ is bounded, then $t \to u_t$ is a bounded turn.
\end{lemma}

\begin{proof} By adding a constant, we can assume that $u_0 \in \mathcal E^1$ satisfies $u_0 \leq 0$. Let $[0,\infty) \ni t \to v^+_t, \  v^-_t \in \mathcal E^1$ be the rays emanating from $v^+_0 = v^-_0=0$  and parallel to $[0,\infty) \ni t \to u_t, \ u_{-t} \in \mathcal E^1$ respectively \cite[Proposition 4.1]{DL20}.

By the construction of $v^+_t$ in the proof of \cite[Proposition 4.1]{DL20} (see also Remark \ref{rem: ray_max}) we have that $u_t \leq v^+_t$ and $u_{-t} \leq v^-_t$. Due to \cite[Lemma 3.2]{DX22} there exists $c>0$ such that $ v^+_t  \leq tc$ and  $v^-_t \leq tc$. We obtain that $u_t \leq c |t|$, finishing the proof.
\end{proof}

We say that test line $\tau \to v_\tau$ is \emph{maximal} if $P[v_\tau](\check v_0) = v_\tau$ for any $\tau \in \Bbb R$. Our next observation is that maximal test lines and geodesic turns are dual:

\begin{prop}\label{prop: turn_RWN} The maps $\{u_t\}_t \to \{\hat u_\tau\}_\tau$ and $\{v_\tau\}_\tau \to \{\check v_t\}_t$ give bijective correspondences between $L^p$ turns and maximal $L^p$ test lines. 
\end{prop}

\begin{proof} Let $t \to u_t$ be an $L^p$ turn and $u^{\pm}_t$ the two geodesic rays emanating from $u_0$ defined in \eqref{eqn: conc_rays}. From Theorem \ref{thm: max_fin_en_test_curve_ray_duality} it follows that $P[u_0](\hat u^+_\tau)=\hat u^+_\tau$ and $P[u_0](\hat u^-_\tau)=\hat u^-_\tau$. 

Observe that $\hat u_\tau = \min(\hat u^+_\tau,\hat u^-_{-\tau})$; this follows from $\tau$-concavity and the formula for $\hat{u}_\tau^{\pm}$ given right after \eqref{eqn: conc_rays}. Then (the proof of) \cite[Lemma 5.2]{DDL5} implies that $P[\hat u_\tau](u_0) = \hat u_\tau$, proving maximality.

Now suppose that $\tau \to v_\tau$ is a maximal $L^p$ test line. We will show that $t \to \check v^+_t$ is an $L^p$ ray; the same proof holds for $t \to \check v^-_t$. The argument is essentially the same as the reverse direction of \cite[Theorem 3.7(ii)]{DX22}. By elementary properties of the Legendre transform we also assume that $\tau^+_v =0$, so that $t \to \check v_t$ is $t$-decreasing. %, hence so is $t \to \check v^+_t$. 

Now assume by contradiction that $t \to \check v^+_t$ is not a psh geodesic ray. Consequently, due to the comparison principle for geodesic segments, there exists $0 \leq a < b$ such that 
\[
\check v^+_{(1-t)a + tb} \lneq \chi_t, \ \  t \in (0,1)\,,
\]
where $[0,1] \ni t \to \chi_t \in \mathcal E^p_\omega$ is the $L^p$ geodesic segment joining $v^+_a$ and $v^+_b$. Now let $t \to \phi_t $ be the sublinear subgeodesic line such that $\phi_t := \check v_t$ for $t \not\in (a,b)$ and $\phi_{a(1-t) + bt} := \chi_t$ otherwise.

Trivially, $\check v_t \leq \phi_t \leq 0$, hence by duality, $v_\tau \leq \hat \phi_\tau$ and $\tau^+_v = \tau^+_{\hat \phi}=0$.  
We claim that  $ \hat \phi_\tau \leq  v_\tau + \tau(a-b)$ for any $\tau \in \mathbb R$. Indeed, since $\tau^+_v = \tau^+_{\hat \phi}=0$, we only need to show this for $\tau \leq 0$. For such $\tau$, we have 
\[
\inf_{t \in [a,b]}(\phi_t - t \tau) \leq \phi_b - b \tau = \check v_b - b \tau  \leq \inf_{t \in [a,b]}(\check v_t - t \tau) +  (a-b) \tau\,,
\]
where in the last inequality we used that $t \mapsto \check v_t$ is decreasing.

Since $\hat \phi_\tau \leq \check v_0$ and $\hat \phi_\tau \leq v_\tau + (a-b)\tau$, by the maximality of $\tau \to v_\tau$, we obtain that $v_\tau = \hat \phi_\tau$. An application of the Legendre transform now gives that $\check v_t=\phi_t$, a contradiction. Hence $t \to v^+_t$ is an $L^p$ geodesic ray, as desired.
\end{proof}

\section{The Ross--Witt Nystr\"om correspondence for lines}

We are ready to prove the first main result in our paper:

\begin{theorem}\label{thm: finite_en_line_char} The maps $\{u_t\}_t \to \{\hat u_\tau\}_\tau$ and $\{v_\tau\}_\tau \to \{\check v_t\}_t$ give bijective correspondences between $L^p$ (resp. bounded) geodesic lines and $L^p$ (resp. bounded) test lines with $\int_X \omega^n_{\hat u_\tau}=0$ for all $\tau \in (\tau^-_u, \tau^+_u)$.
\end{theorem}

We point out that the zero mass condition can not be extended to the endpoints of the interval $[\tau^-_u, \tau^+_u]$ in the above theorem. This is immediately seen for a constant ray $t \to u_t = u \in \mathcal H_\omega$, whose Legendre transform satisfies $\hat u_0 = u $ and $\hat u_\tau \equiv -\infty$ otherwise. However for bounded test lines with $\tau^-_u< \tau^+_u$,  this property does hold per Remark \ref{rem: endpoints}.

A consequence of our argument below is that $L^p$ geodesic lines are always $t$-Lipschitz; moreover their  test lines are always maximal. We will only argue the $L^p$ case of the above theorem. Once this is done, the bounded case is a simple argument left to the reader.

\begin{proof} First we suppose that $t \to u_t$ is a non-constant $L^p$ geodesic line. Due to Proposition \ref{prop: turn_RWN}, $\tau \to \hat u_\tau$ is an $L^p$ test line, and we only have to argue the zero mass property. We fix $\tau \in (\tau^-_u, \tau^+_u)$.

As geodesic lines are also geodesic turns, we obtain that $P[\hat u_\tau](u_0) = \hat u_\tau$ from Proposition \ref{prop: turn_RWN}. In particular, by \cite[Theorem 3.8]{DDL2}, $\omega_{\hat u_\tau}^n$ is supported on $\{\hat u_\tau = u_0\}$. Fixing $t_0 \in \Bbb R$, we introduce the `shifted' geodesic line $v_t:= u_{t + t_0}$. Since, $\hat v_\tau = \inf_{t \in \Bbb R }(u_{t + t_0} - t \tau) = \hat u_\tau + t_0 \tau$, by the same reasoning, 
$\omega_{\hat u_\tau}^n = \omega_{\hat v_\tau}^n$ is supported on the contact set $\{\hat v_\tau = v_0\} = \{\hat u_\tau = u_{t_0} - {t_0} \tau\}$. 

As $t_0 \in \Bbb R$ is arbitrary and $\tau$ is fixed, 
we obtain that 
$$\textup{supp }  \omega_{\hat u_\tau}^n  \subseteq \bigcap_{t_0 \in \Bbb R} \{\hat u_\tau = u_{t_0} - {t_0} \tau\}.$$ 
Now, fix $\tau'\in(\tau_u^-, \tau_u^+)$, $\tau'\not= \tau$ and let $x\in \bigcap_{t_0 \in \Bbb R} \{\hat u_\tau = u_{t_0} - {t_0} \tau\}$. If $u_0(x) > -\infty$ then we have $u_{t_0}(x) = u_0(x) + t_0 \tau$ for all $t_0\in\R$, implying that $\hat{u}_{\tau'}(x) = -\infty$. Since we also clearly have $\{u_0 = -\infty\}\subseteq \{\hat{u}_{\tau'} = -\infty\}$, we conclude that $\bigcap_{t_0 \in \Bbb R} \{\hat u_\tau = u_{t_0} - {t_0} \tau\} \subseteq \{\hat{u}_{\tau'} = -\infty\}$; it follows that $\omega_{\hat{u}_\tau}^n$ is supported on a pluripolar set, and thus vanishes.

\smallskip

Next,  let $\tau \to \hat u_\tau$ be a zero mass $L^p$ test line, i.e., $\int_X \omega^n_{\hat u_\tau}=0$ for $\tau  \in (\tau^-_u, \tau^+_u)$.  Due to Proposition \ref{prop: RWN_sublinear_subgeod_lines}, $t \to u_t$ is automatically an $L^p$ sublinear subgeodesic line. To finish the proof,  we need to show that $t \to u_t$ is in fact an $L^p$ geodesic line.

Let $u^{\pm}_t$ be the two $L^p$ subgeodesic rays from \eqref{eqn: conc_rays}. To show that $t \to u_t$ is a line, we will argue that
\begin{equation}\label{eq: I_slope_eq}
I\{u^+_t\} := \lim_{t \to \infty} \frac{I(u^+_t)}{t} = - I\{u^-_t\} := \lim_{t \to \infty} \frac{I(u^-_t)}{t}.
\end{equation}
Indeed, this would imply that the convex function $t \to I(u_t), \ t \in \Bbb R$, is in fact affine; hence $t \to u_t$ is a geodesic line, as desired.

Using the formula for $\hat{u}_{\tau}^{\pm}$ immediately after \eqref{eqn: conc_rays}, we see that $\hat{u}_\tau^+ = u_0$ for $\tau \leq \tau_u^-$ and $\hat{u}_\tau^- = u_0$ for $\tau \leq -\tau_u^+$. Since $\int_X \omega_{u_0}^n = V$, we deduce from Corollary \ref{cor: maximiz_I_slope_for} and Lemma \ref{lem: tau_line_ray} that:
\[
I\{u^-_t\} = -\tau^-_{u} +\frac{1}{V} \int^{-\tau^-_{u}}_{-\tau_u^+} \Big( -V +\int_X \omega^n_{\hat u^-_\tau}\Big)d\tau = -\tau^+_{u} +\frac{1}{V} \int^{\tau^+_{u}}_{\tau_u^-} \int_X \omega^n_{\hat u^-_{-\tau}}\, d\tau\ \text{ and }
\]
\[
I\{u_t^+\} = \tau^+_u+\frac{1}{V} \int_{\tau_u^-}^{\tau^+_u} \Big(-V  + \int_X \omega^n_{\hat u^+_\tau}\Big)d\tau = \tau^-_u+\frac{1}{V} \int_{\tau_u^-}^{\tau^+_u}\int_X \omega^n_{\hat u^+_\tau}\, d\tau.
\]
By convexity of $t\to I(u_t)$, we have $-I\{u^-_t\} \leq I\{u^+_t\}$. Using \cite[Theorem 1.2]{DDL5}, we see that:
\begin{flalign}
0 \leq I\{u^+_t\} + I\{u^-_t\} &= \frac{1}{V}  \int_{\tau^-_u}^{\tau^+_u}\Big(\int_X \omega^n_{\hat u^-_{-\tau}} + \int_X \omega^n_{\hat u^+_{\tau}}\Big)d\tau - (\tau^+_u - \tau^-_u) \nonumber \\
& \leq \frac{1}{V} \int_{\tau^-_u}^{\tau^+_u}\Big(\int_X \omega^n_{P(\hat u^-_{-\tau}, \hat u^+_{\tau})} + \int_X \omega^n_{\max(\hat u^-_{-\tau}, \hat u^+_{\tau})}\Big)d\tau - (\tau^+_u - \tau^-_u).\label{eq: int_ineq}
\end{flalign}
Now, by $\tau$-convexity and the formula for $\hat{u}^{\pm}_\tau$, we see that $\max(\hat u^+_{\tau}, \hat u^-_{-\tau})=u_0 \in \mathcal{E}^p_\omega$ and $P(\hat u^-_{-\tau}, \hat u^+_{\tau}) = \min(\hat u^-_{-\tau}, \hat u^+_{\tau}) = \hat u_\tau$, so that:
$$\int_X \omega^n_{P(\hat u^-_{-\tau}, \hat u^+_{\tau})}=\int_X \omega^n_{\hat u_\tau}=0\ \ \text{ and }\ \ \int_X \omega^n_{\max(\hat u^-_{-\tau}, \hat u^+_{\tau})} = \int_X \omega^n = V.$$ 
Putting this into \eqref{eq: int_ineq} we obtain:
$$0 \leq  \frac{1}{V}  \int_{\tau^-_u}^{\tau^+_u}\Big(\int_X \omega^n_{\hat u^-_{-\tau}} + \int_X \omega^n_{\hat u^+_{\tau}}\Big)d\tau -(\tau^+_u - \tau^-_u) \leq (\tau^+_u - \tau^-_u)-(\tau^+_u - \tau^-_u) =0.$$
Hence we obtain \eqref{eq: I_slope_eq}, what we needed to prove. 
\end{proof}
In the last step of the above argument we proved the following identity for $L^p$ geodesic lines $t \to u_t$:% that we note for future use:
\begin{equation}\label{eq: volume_eq}
\int_X \omega^n_{\hat u^-_{-\tau}} + \int_X \omega^n_{\hat u^+_{\tau}} = \int_X \omega^n,  \ \ \ \ \tau \in \Bbb R, \ a.e.,
\end{equation}
where we used the convention $\int_X \omega_{-\infty}^m =0$. From \cite[Lemma 3.9]{DX22} we obtain that the function $\tau \to \int_X \omega_{\hat u^+_\tau}^n$ is continuous on $(-\infty,\tau^+_u)$, and the function $\tau \to \int_X \omega_{\hat u^-_{-\tau}}^n$ is continuous  on $(\tau^-_u,\infty)$. From this it follows that the above identity holds for all $\tau \in \Bbb R$.%, but we will not need this fact.

More generally, for any sublinear, bounded (or $L^p$) subgeodesic line  $t \to u_t$ $\tau \to u_\tau$, it is natural to ask if
\[
\int_X \omega^n_{\hat u^-_{-\tau}} + \int_X \omega^n_{\hat u^+_{\tau}} = \int_X \omega^n_{u_\tau} + \int_X \omega^n,  \ \ \ \ \tau \in \Bbb R, \ a.e.
\]
Such an equality would have intrinsic interest in producing non-trivial examples of equality in the volume diamond inequality of \cite[Theorem 1.3]{DDL5}.

\begin{remark} \label{rem: endpoints} Let $\{v_\tau\}_\tau$ be a bounded zero mass test line with $\tau^-_v < \tau^+_v$. Then $v_{\tau^+_v},v_{\tau^-_v}$ are not identically equal to $-\infty$ and are of zero mass.
\end{remark}

\begin{proof} We start with showing that $v_{\tau^+_v},v_{\tau^-_v}$ are not identically equal to $-\infty$. 
By replacing $\{v_\tau\}_\tau$ with $\{v_{-\tau}\}_\tau$ it is enough to argue this for $v_{\tau^+_v}$. Moreover, by composing with an affine $\tau$-transformation, we can assume that $\tau^-_v=-1$ and $\tau^+_v=0$. This means that the bounded geodesic line $t \to \check{v}_t$ is $t$-decreasing. 

We use \cite[Theorem 1]{Da17}(ii) to conclude that $\sup_X (\check v_t - \check v_0) = C t$ for $t \geq 0$ for some $C \leq 0$. Since, $0= \tau^+_v \leq C \leq 0$, we must have that $C=0$. This implies that $\sup_X \check v_t, \ t \geq 0$ is uniformly bounded. Hence the decreasing set of potentials $\{\check v_t\}_{t \geq 0}$ must have an $L^1$-limit, not identically equal to $-\infty$. By definition, this limit is equal to $v_{\tau^+_v}=v_{0}$.

To argue that both $v_{\tau^+_v},v_{\tau^-_v}$ are of zero mass, we use $\tau$-concavity, monotonicity of volumes \cite[Theorem 1.2]{WN19}, and multlinearity of non-pluripolar products:
$$0 \leq \frac{1}{2^n}\int_X  \omega_{v_{\tau^+_v}}^n + \frac{1}{2^n}\int_X  \omega_{v_{\tau^-_v}}^n \leq \int_X  \omega_{\frac{1}{2}v_{\tau^+_v} + \frac{1}{2}v_{\tau^-_v}}^n \leq \int_X  \omega_{v_{\frac{1}{2} \tau^-_v + \frac{1}{2} \tau^+_v}}^n  =0.$$
\end{proof}

\section{Examples of geodesic lines}\label{sec: examples}

In this section, we use Theorem \ref{thm: RWN_lines} to construct various examples of geodesic lines. We will start by giving a general criterion that will be used many times:

\begin{prop}\label{prop: gen_crit} Let $v_0,v_1 \in \PSH_\omega$ so that each of the non-pluripolar currents $\omega_{v_0}^k, \ k \in \{1,\ldots,n\}$ and $\omega_{v_1}^n$ vanish, and $\max (v_0,v_1)$ is bounded. Additionally, let $g: [0,1] \to \Bbb R$ be a continuous concave function. Then the following curve $\tau \to v_\tau$ is a bounded zero mass test line 
$$v_\tau = (1-\tau)v_0 + \tau v_1 + g(\tau), \ \ \ \tau \in [0,1], \ \ \ v_\tau = -\infty, \ \ \tau \in \Bbb R \setminus [0,1].$$
\end{prop}
\begin{proof} To check the zero mass property, we use the  multilinearity of non-pluripolar products for $\tau \in [0,1]$:
\[
\omega_{v_\tau}^n = \sum_{k=0}^n {n\choose k} \tau^k (1-\tau)^{n-k} \omega_{v_0}^k\wedge \omega_{v_1}^{n-k} = 0,
\]
where we used that $\omega_{v_0}^k = 0$, for any $k$, and $\omega_{v_1}^n = 0$. The other defining properties of a test line are straightforward. 
\end{proof}

The above criterion provides us the weak geodesic lines of Example \ref{ex: proj}:

\begin{coro}\label{cor: proj_2}
Let $L \to X$ be an ample line bundle such that $c_1(L) = \{\omega\}$. Let $D_0,\ldots, D_n$ be divisors in the class $c_1(L)$ such that $\cap_j D_j = \emptyset$. Let $\varphi_j \in \textup{PSH}_\omega$ such that $\omega_{\varphi_j} = [D_j]$. %\medskip\\
%\noindent (i) 
\begin{equation}\leqnomode
u_t := 
\max(\varphi_0, \varphi_1 + t,\ldots, \varphi_n + t) \in \textup{PSH}_\omega \cap C^{0,1}, \ t \in \Bbb R\tag{i}
\end{equation}
defines a $C^{0,1}$ geodesic line. %\medskip\\
%\noindent (ii)
\begin{equation}\leqnomode
v_t := \log ( e^{\varphi_0} + \sum_{j =1}^n  e^{\varphi_j + t}) \in \textup{PSH}_\omega \cap C^\infty, \ t \in \Bbb R\tag{ii}
\end{equation}
defines a $C^{\infty}$ geodesic line, which is parallel to $t\to u_t$.
\end{coro}
\begin{proof} We deal with (i).  Let $\hat u_0 := \varphi_0$ and $\hat u_1 := \max(\varphi_1, \ldots, \varphi_n)$. First, we show that $\hat{u}_1$ has zero mass, by arguing as in \cite[Example 6.11, p. 165]{Dem12}. It will suffice to show that $\omega_{\hat{u}_1}^n$ vanishes on $\Omega= X \setminus \big(\cup_j D_j\big)$.

Fix a hermitian metric $h$ on $L$ such that $\omega = R_h$, the curvature form of $h$. Let $s_j$ be defining sections for $D_j$, $j = 1,\ldots, n$. After possibly taking powers of $L$, we can assume that $s_j \in H^0(X,L)$. Then there exist constants $a_j\in\Bbb R$ such that:
\[
\varphi_j = a_j + \log |s_j|^2_h \quad\text{ on }X.
\]
After rescaling $s_j$ we can assume that $a_j=0$. Since $s_1$ does not vanish on $\Omega$, $\frac{s_j}{s_1}$ is holomorphic there and:
\[
\hat{u}_1 = \varphi_1 + \max_{j > 1}\left\{0,\log \left|\frac{s_j}{s_1}\right|^2\right\};
\]
it follows that $\omega_{\hat{u}_1}$ only depends on $n-1$ holomorphic functions on $\Omega$. This forces the top form $\omega_{\hat{u}_1}^n$ to vanish on $\Omega$, as desired.

We introduce the linear curve:
\[
\hat{u}_\tau := (1-\tau)\vp_0 + \tau \max_{1\leq j\leq n} \vp_j = (1-\tau)\hat{u}_0 + \tau\hat{u}_1\quad\tau\in[0,1],
\]
and $\hat u_\tau =-\infty$ otherwise. since $\omega_{\hat{u}_0}^k = 0$, for any $k$ we can apply Proposition \ref{prop: gen_crit} to conclude that $\tau \to \hat u_\tau$ is a bounded zero mass test line. We conclude by Theorem \ref{thm: RWN_lines} that $t \to u_t$ is indeed a $C^{0,1}$ geodesic line. 

We argue part (ii) similarly; in this case, the computation of the Legandre transform $\hat v_\tau$ is slightly more intricate, though still elementary. 

Indeed, we leave it to the reader to verify that for any $a, b > 0$, the Legendre transform of $t\to \log (a + b e^t), t \in \Bbb R$ is $-\infty$ on $\Bbb R \setminus [0,1]$ and on the interval $[0,1]$ it is equal to
\[
(1-\tau) \log(a) + \tau \log(b) -\tau \log \tau - (1-\tau) \log (1-\tau), \quad \tau\in [0, 1].
\]
It follows that
$$\hat v_\tau = (1-\tau) \vp_0 + \tau \log \Big( \sum_{j=1}^n e^{\vp_j}  \Big) -\tau \log \tau - (1-\tau) \log (1-\tau), \ \ \tau \in [0,1]$$
holds at all points where none of the $s_j$ vanish; since both the left and right hand side are qpsh functions however, they must agree on all of $X$.

Finally, observe that for all $\tau\in [0,1]$, $\hat{u}_\tau \leq \hat{v}_\tau \leq \hat{u}_\tau + 2e^{-1} +\log n$, so by \cite[Theorem 1.2]{WN19}, $\tau\to \hat{v}_\tau$ is also a zero mass test line. That $t\to v_t$ and $t\to u_t$ are parallel follows from the same argument as in Example \ref{ex: paral_line}.
\end{proof}

We take the opportunity to point out an elementary direct argument for the proof of the above result that does not involve the use of Theorem \ref{thm: RWN_lines}. As the argument is virtually the same for both parts (i) and (ii) we only deal with the former case. 

Let $u(s,x) = u_{\textup{Re }s}(x)$ for $x \in X$ and $s \in \Bbb C$. We need to show that 
\begin{equation}\label{eq: geod_eq}
(\pi^* \omega + dd^c u)^{n+1} =0 \textup{ on } \Bbb C \times X.
\end{equation}
It is enough to check that this equation holds locally, near an arbitrary point $(s,x) \in \Bbb C \times X$. So let $U \subset X$ be an open neighborhood of $X$ trivializing $L$. As as a result, we can assume that $h \in C^\infty(U)$ is a positive function, $s_j \in \mathcal O(U)$ and $\omega = -d d^c h$. One of the $s_j$ does not vanish at $x$. We assume that  $s_0(x) \neq 0$; the argument is virtually identical if $s_j(x) \neq 0$ for some $j\geq 1$.

On $\Bbb C \times U$ we can expand \eqref{eq: geod_eq} as follows:
\begin{flalign*}
(\pi^* \omega + dd^c u)^{n+1} &= (-dd^c h + dd^c u)^{n+1}  = dd^c (\log \max( \log |s_0|^2, \log |s_1 e^s|^2, \ldots,\log |s_n e^s|^2) ^{n+1}\\
& = dd^c  \Big(\max \Big ( 0, \log \Big|\frac{s_1 e^s}{s_0}\Big|^2, \ldots,\log \Big|\frac{s_n e^s}{s_0}\Big|^2 \Big) ^{n+1}
\end{flalign*}
Since the expression inside the $dd^c$ depends on $n$ holomorphic functions only, the $(n+1)$-intersection above must vanish.
\medskip

On Riemann surfaces many more examples of lines can be provided, that do not arise from algebraic data:

\begin{example}\label{ex: geod_line_surface}
Let $(X,\omega)$ be a Riemann surface. Let $\mu_1,\mu_2$ be two Borel measures supported on polar sets $C_1,C_2$ with disjoint closure. Assume that $\int_X d\mu_1 = \int_X d\mu_2  = \int_X \omega$. Typical examples might be:
\begin{itemize}
 \item when $C_1, C_2$ are discrete and $\mu_1, \mu_2$ are sums of Dirac masses, or
 \item when $C_1, C_2$ are polar Cantor sets and $\mu_1, \mu_2$ are the distribution measures of the corresponding Cantor functions (see \cite[Section 3]{Li}, \cite{Ra} for the construction).
\end{itemize}

By convolving with the Green's function for $\omega$ (explained in detail in \cite[Section 3]{Li}), there exists $u_0,u_1 \in \textup{PSH}_\omega$ such that
$$1 + \Delta_\omega u_j = \mu_j\ \ \ j = 0, 1.$$
The $u_j$ are smooth away from the closure of $C_j$; when the $C_1, C_2$ are discrete or are closed Cantor sets, we also have $\{u_j = -\infty\} = C_j$ (\cite[Proposition 3.3]{Li}). Moreover, the non-pluripolar measures $\omega_{u_j}$ must vanish on $X$. 
Let 
$$u_{\tau} := \tau u_1 + (1-\tau) u_0, \ \ \tau \in [0,1];$$ 
otherwise set $u_\tau \equiv -\infty$. Clearly this is a bounded, zero mass test line, so by Theorem \ref{thm: finite_en_line_char}, the Legendre transform:%This curve is a test line with corresponding sublinear subgeodesic
$$\check u_t = \sup_{\tau \in [0,1]} (u_\tau + t\tau) = \sup_{\tau \in [0,1]} (\tau u_1 + (1-\tau) u_0 + t\tau) = \max (u_0, u_1 + t), \ \ t \in \Bbb R$$
is a bounded geodesic line. 
When $C_1, C_2$ are discrete, or closed Cantor sets, then $t \to \check u_t$ is additionally a $C^{0,1}$ geodesic line.
\end{example}

Using Proposition~\ref{prop: gen_crit}, we now show that on Riemann surfaces -- modulo the choice of a concave function -- the above construction yields almost all bounded weak geodesic lines:

\begin{theorem}\label{prop: class_RS}
Suppose that $(X,\omega)$ is a Riemann surface, and that $\tau \to v_\tau$ is a bounded, zero mass test line. Then there exists a continuous concave function $g(\tau)$ on $[\tau_{v}^-, \tau_{v}^+]$ such that $\tau\mapsto v_\tau(x) - g(\tau)$ is linear for each $x\in X$ and $\tau \in [\tau_{v}^-, \tau_{v}^+]$.
\end{theorem}
\begin{proof}

By elementary properties of the Legendre transform, we may assume that $\tau_v^- = 0$ and $\tau_{v}^+ = 1$. Further, by Remark \ref{rem: endpoints}, both $v_0$ and $v_1$ are zero-mass potentials, so we can consider the linear curve:
\[
u_\tau := (1-\tau)v_0 + \tau v_1 \quad \tau\in[0,1];
\]
clearly each $u_\tau$ will be a zero-mass potential. By concavity, we also have that $u_\tau \leq v_\tau.$
Since $X$ is a Riemann surface and both sides are zero-mass, \cite[Lemma 4.1]{ACCP} now implies that they must be equal up to a constant. We deduce the existence of a function $g(\tau)$ such that:
\[
(1-\tau)v_0 + \tau v_1 = v_\tau - g(\tau) \quad\text{ for all }\tau\in[0,1].
\]
Since $\tau \to v_\tau$ is $\tau$- usc and $\tau$-concave, we obtain that $g(\tau)$ is concave and continuous, finishing the proof.
\end{proof}

In higher dimensions, one cannot hope for such a simple classification of weak geodesic lines:

\begin{example}\label{ref: not_Riemann} Suppose that $\{u_t\}_t, \{v_t\}_t$ are geodesic lines on $(X,\omega_X), (Y,\omega_Y)$; due to the splitting of the Monge--Amp\`ere measure on products one can see that $\{u_t + v_t\}_t$ is a geodesic line on $(X\times Y, \omega_X + \omega_Y)$. The corresponding test line is given by the ``supremal convolution" \cite[Theorem 16.4, p. 145]{Ro70}
\[
\hat{w}_\tau := \sup_{\sigma\in\R} \{\hat{u}_{\tau - \sigma} + \hat{v}_{\sigma}\}.
\]
In particular, if $\hat{u}_\tau = (1-\tau)\vp_0 + \tau\psi_0$ and $\hat{v}_\tau = (1-\tau)\vp_1 + \tau\psi_1$, $\tau\in[0,1]$, then:
\[
\hat{w}_{\tau} = \begin{cases}
(1-\tau)(\vp_0 + \vp_1) + \tau\max\{\vp_0 + \psi_1, \vp_1 + \psi_0\} &\text{ if }0\leq \tau\leq 1\\
(\tau-1)(\psi_0 + \psi_1) + (2-\tau)\max\{\vp_0 + \psi_1, \vp_1 + \psi_0\} &\text{ if }1\leq \tau\leq 2.
\end{cases}
\]
\end{example}

\section{Existence of parallel lines}

We say that two geodesic lines (resp. rays) $t \to v_t,w_t$ are \emph{parallel} if $t \to d_p(v_t,w_t)$ is bounded on $\Bbb R$ (resp. $\Bbb R^+$) \cite{DL20}. Recall that the function $t \to d_p(v_t,w_t)$ is convex due to \cite[Theorem 1.5]{CC3}. Consequently, for parallel lines, $t \to d_p(v_t,w_t)$ is constant, while for parallel rays, $t \to d_p(v_t,w_t)$ is decreasing.

Given a ray $t \to v_t$ in $\mathcal E^p_\omega$, there exists a unique parallel geodesic ray $t \to v_t^h$  emanating from any outside point $h \in \mathcal E^p_\omega$ \cite[Proposition 4.1]{DL20}. Moreover, it follows from the proof of \cite[Proposition 4.1]{DL20} that the parallel rays satisfy a monotonicity property: if $h \leq h'$ then $v^h_t \leq v^{h'}_t$ for all $t \geq 0$.

We now prove our theorem regarding existence of parallel geodesic lines.

\begin{theorem}\label{thm: 5th_post_1} Let $p \geq 1$. Suppose that $\Bbb R \ni t \to v_t \in \PSH_\omega \cap L^\infty$ is a weak geodesic line and $u \in \textup{PSH}_\omega \cap L^\infty$. There exists a unique geodesic line $\Bbb R \ni t \to u_t  \in \PSH_\omega \cap L^\infty$ with $u_0 = u$ and $d_p(u_t,v_t) = d_p(u_0,v_0)$ for any $t \in \Bbb R$ if and only if
\begin{equation}\label{eq: cond_5th_post_1}
u = \sup_{\tau \in \Bbb R} P[\hat v_\tau](u).
\end{equation}
Moreover, if \eqref{eq: cond_5th_post_1} holds then $\hat u_\tau = P[\hat v_\tau](u), \  \tau \in \Bbb R$.
\end{theorem}

\begin{proof} First we argue uniqueness. Assume that there are two parallel lines $\Bbb R \ni t \to u'_t,u_t \in \mathcal E^p$ with $u'_0 = u_0=u$. These lines are parallel to each other as well. By the comments preceding the theorem, we get that $d_p(u'_t,u_t) = d_p(u'_0,u_0)=0, \ t \in \Bbb R$, proving that these lines are indeed the same.

After adding a constant we can assume $v_0 \leq u$. Moreover, there exists $C>0$ such that 
\begin{equation}\label{eq: v_0_u_assumpt}
v_0 \leq u \leq v_0 + C
\end{equation}

First suppose that the parallel line $t \to u_t$ in the statement of the theorem exists. As a result, the rays $t \to v^\pm_t, u^\pm_t, v^\pm_t + C$ are parallel to each other. Using the monotonicity property of parallel rays (recalled before the theorem) we obtain that $ v^\pm_t\leq u^\pm_t \leq v^\pm_t + C$. Consequently, $ v_t\leq u_t \leq v_t + C$ and after taking Legendre transforms we obtain that
\begin{equation}\label{eq: u_v_eqv}
\hat v_\tau \leq \hat u_\tau \leq \hat v_\tau + C, \ \tau \in \Bbb R.
\end{equation}
By the maximality property of these potentials (Proposition \ref{prop: turn_RWN}), we get that \begin{equation}\label{eq: }
\hat u_\tau = P[\hat u_\tau](u_0) = P[\hat u_\tau](u)=P[\hat v_\tau](u), \ \tau \in \Bbb R.
\end{equation}
This proves that last statement of the Theorem. By inverse duality, we obtain that $u = u_0 = \sup_{\tau \in \Bbb R} \hat u_\tau = \sup_{\tau \in \Bbb R} P[\hat v_\tau](u_0)$, proving \eqref{eq: cond_5th_post_1}, as desired.

Now we suppose that \eqref{eq: cond_5th_post_1} holds. We introduce the following `candidate' test line: 
$$\tilde u_\tau:= P[\hat v_\tau](u), \ \tau \in \Bbb R.$$
Since $\tau \to \hat v_\tau$ is $\tau$-concave, by definition of $P[\hat v_\tau](u)$ it follows that so is $t\to \hat u_\tau$.

Since $P[\hat v_\tau](v_0)=\hat v_\tau $ (Proposition \ref{prop: turn_RWN}) and $u$ is bounded, it follows that $[\hat v_\tau] = [\tilde u_\tau]$ and and \cite[Theorem 1.2]{WN19} implies that  is $\tau \to \tilde u_\tau$  is zero mass.  However it is not immediately clear that $\tau \to \tilde u_\tau$ is also $\tau$-usc. To deal with this, we will need to take some sort of $\tau$-usc envelope of $\tau \to \tilde u_\tau$ and work with that.
To start, we apply the inverse Legendre transform:
$$\phi_t:= \sup_{\tau \in [\tau^-_v,\tau^+_v]} (\hat u_\tau + t \tau).$$
Let $\phi(s,z) := \phi_{\textup{Re }s}(z)$. By Choquet's lemma, $\textup{usc }\phi \in \PSH_{\pi^*\omega}(\Bbb C \times X)$, hence $t \to \textup{usc }\phi(t,\cdot)$ is a subgeodesic line, but it is not clear if $\textup{usc }\phi = \phi$.

Let $S = \{\phi < \textup{usc }\phi\}$. Since $\phi$ and $\textup{usc }\phi$ are $i\Bbb R$ invariant, so is $S  \subset \Bbb C \times X$, hence $S = S_\Bbb R  \times i \Bbb R$, where $S_\Bbb R \subset\Bbb  R \times X$. 

By definition $S$ is pluripolar, hence it and $S_{\Bbb R}$ have Lebesgue measure zero. Given $x \in X$, let $S_x = S_\Bbb R \cap \Bbb R \times \{x\}$ be the $x$-slice of $S_\Bbb R$. By Fubini's theorem, there exists $F \subset X$ of Lebesgue measure zero, such that $S_x \subset \Bbb R$ has Lebesgue measure zero for all $x \in X \setminus F$.

Due to \eqref{eq: v_0_u_assumpt} we get that 
\begin{equation}\label{eq: v_0_u_assumpt_1}
v_t \leq \phi(t,x) \leq \textup{usc }\phi(t,x) \leq v_t + C, \ t \in \Bbb R,
\end{equation}
hence both $t \to \phi(t,x)$ and $t \to \textup{usc }\phi(t,x)$ are convex, finite, and continuous. Consequently, $S_x$ is actually empty (not just measure zero) for all $x \in X \setminus F$.

Let $\chi_\tau := \inf_{t \in \Bbb R} ( {\textup{usc } \phi}(t,\cdot) - t \tau).$ Since $t \to {\textup{usc } \phi}(t,\cdot)$ is a sublinear subgeodesic line, we get that $\tau \to \chi_\tau$ is a test line with $\tau^-_\chi = \tau^-_v$ and $\tau^+_\chi = \tau^+_v$ (Proposition \ref{prop: RWN_sublinear_subgeod_lines}). 

We next argue that $\chi_\tau = \tilde u_\tau$ for $\tau \in (\tau^-_\chi,\tau^+_\chi)$ making $\tau \to \chi_\tau$ a zero mass test line.
Indeed, we fix $\tau_0 \in (\tau^-_\chi,\tau^+_\chi)$. Let $\varepsilon>0$ such that $(\tau_0-\varepsilon,\tau_0+\varepsilon) \subset (\tau^-_\chi,\tau^+_\chi)$.

Let $G := \{\chi_{\tau_0 - \varepsilon} = -\infty\} \cup\{\chi_{\tau_0 + \varepsilon} = -\infty\}$. This set $G$ also has zero Lebesgue mass, hence if $x \in X \setminus (F \cup G)$ then $\tau \to \tilde u_\tau(x)$ is concave and finite on $(\tau-\varepsilon,\tau+\varepsilon)$, hence also continuous. Due to Legendre duality, for any $x \in X \setminus (F \cup G)$ fixed, the function $\tau \to \chi_\tau(x)$ is the $\tau$-usc envelope of $\tau \to \tilde u_\tau(x)$. These last two facts imply that for $x \in X \setminus (F \cup G)$ we have $\chi_{\tau_0}(x) = u_{\tau_0}(x)$. Since both $\chi_{\tau_0}$ and $u_{\tau_0}$ are qpsh and $F \cup G$ has Lebesgue measure zero, we get that $\chi_{\tau_0}(x) = u_{\tau_0}(x)$ for all $x \in X$.
 
By Theorem \ref{thm: RWN_lines} we obtain that $t \to \textup{usc }\phi(t, \cdot)$ is a bounded weak geodesic line. Due to \eqref{eq: v_0_u_assumpt_1} we obtain that  $|\textup{usc }\phi(t, \cdot) - v_t| \leq C, t \in \Bbb R$, hence $t \to d_p(\textup{usc }\phi(t, \cdot),v_t)$ is bounded \cite[Theorem 5.5]{Da15}. Due to \eqref{eq: cond_5th_post_1} we obtain that $\phi_0 = u$. Note that $ u - C |t| \leq \phi_t \leq u + C |t|$ for some $C>0$. Consequently, also $u - C |t| \leq \textup{usc } \phi(t,\cdot) \leq u + C |t|$. In particular, $\textup{usc } \phi(0,\cdot) =u$, hence $t \to \phi(0,\cdot)$ is the parallel geodesic line we sought to construct.
\end{proof}

We conclude by showing that condition \eqref{eq: cond_5th_post_1} does not hold for arbitrary rays $t \to v_t$ and potentials $u$.

Let $(X,\omega)$ be a Riemann surface with $\int_X \omega =1$ and $t\to v_t$ a (non-trivial) bounded geodesic line; after normalizing, we may assume $\tau_v^- = 0$ and $\tau_v^+ = 1$. Then by Proposition \ref{prop: class_RS}, there exists a bounded concave function $g(\tau)$, $\tau\in [0, 1]$, such that $\hat{v}_t = (1-\tau)\hat v_0 + \tau\hat v_1 + g(\tau)$. Since we are only interested in lines parallel to $t\to v_t$, we may also assume that $\sup_X \hat{v}_0 = \sup_X \hat{v}_1 = 0$ and $\hat{v}_0\not=\hat{v}_1$.

Consider the `candidate' test line $\tau\to P[\hat{v}_\tau](0)$, coming from setting $u = 0$. By \cite[Lemma 4.1]{ACCP}, we have $P[\hat{v}_\tau](0) = \hat{v}_\tau + h(\tau)$ for some bounded function $h(\tau)$, $\tau\in[0,1]$. Since $\sup_X P[\hat{v}_\tau](0) =0$, we must have $h(\tau) = -g(\tau) - \sup_X[(1-\tau)\hat{v}_0 + \tau\hat{v}_1]$.

If the geodesic line defined by $\tau \to P[\hat{v}_\tau](0)$ were to pass through 0, then by Theorem \ref{thm: 5th_post_1} we must have:
\begin{flalign}\label{eq: u_cond}
0 = \sup_{\tau \in [0,1]}P[\hat v_\tau](0)(x) &= \sup_{\tau \in [0,1]} (1-\tau) \hat v_0(x) + \tau \hat v_1(x) - \sup_X [(1-\tau)\hat{v}_0 + \tau\hat{v}_1]
\end{flalign}
for all $x\in X$. In particular, we see that $\{\hat{v}_0 = -\infty\} \subseteq \{\hat{v}_1 = 0\}$, which will generally not be true. It follows that 
the fifth postulate does not hold for $(\mathcal E^p_\omega, d_p)$.

\subsection*{Declarations}

{\noindent \bf Financial Interests:} The first author was partially supported by a Simons Fellowship and NSF grant DMS-2405274. This material is based upon work supported by the NSF under grant DMS-1928930, while the first author was in residence at the Simons Laufer Mathematical Sciences Institute (formerly MSRI) during the Fall of 2024.\\

{\noindent \bf Non-Financial Interests:} None.

\subsection*{Data Availability Statement:} There is no data associated with this work.

\footnotesize
\let\OLDthebibliography\thebibliography % squeezes Bibliography
\renewcommand\thebibliography[1]{
  \OLDthebibliography{#1}
  \setlength{\parskip}{1pt}
  \setlength{\itemsep}{1pt}
}

\normalsize
\noindent {\sc Department of Mathematics, University of Maryland}\\
{\tt tdarvas@umd.edu}\vspace{0.1in}

\noindent {\sc Department of Mathematics, Purdue University}\\
{\tt  nmccleer@purdue.edu}\vspace{0.1in}
\end{document}